\documentclass{amsart}
\usepackage{graphicx}
\usepackage{mathrsfs}
\usepackage{color}
\usepackage[all]{xy}
\usepackage{verbatim}
\usepackage{amssymb}
\usepackage[colorlinks,
            linkcolor=red,
            anchorcolor=green,
            citecolor=blue]{hyperref}
\newtheorem{thm}{Theorem}[section]
\newtheorem{cor}[thm]{Corollary}
\newtheorem{cla}[thm]{Claim}
\newtheorem{lem}[thm]{Lemma}

\newtheorem{prop}[thm]{Proposition}
\theoremstyle{definition}
\newtheorem{defn}[thm]{Definition}
\newtheorem{rem}[thm]{Remark}
\newtheorem{var}[thm]{Variant}
\numberwithin{equation}{section}
\begin{document}

\title[]{Vanishing for Hodge ideals of $\mathbb{Q}$-divisors}%
\dedicatory{}
\author{Bingyi Chen}
\address{Department of Mathematical Sciences,
Tsinghua University,
Beijing, 100084, P. R. China.}
\email{chenby16@mails.tsinghua.edu.cn}

\thanks{}%

%\date{}%
%\dedicatory{}%
%\commby{}%
% ----------------------------------------------------------------
\begin{abstract}
Musta\c t\u a  and Popa \cite{MP19a} introduce the notion of Hodge ideals for an effective $\mathbb{Q}$-divisor $D$ and prove a vanishing theorem for Hodge ideals, which generalizes Nadel vanishing for multiplier ideals. However, their proof needs an extra assumption on the existence of $\ell$-roots of the line bundle $\mathscr{O}_X(\ell D)$, which is not necessary for  Nadel vanishing. In this paper, we prove that vanishing for Hodge ideals still holds even without this assumption.
\end{abstract}
\maketitle

% ----------------------------------------------------------------

\section{introduction}
Let $X$ be a smooth complex variety and $Z$ a reduced integral effective divisor on $X$. Denote by $\mathscr{O}_X(*Z)$ the sheaf of rational functions on $X$ with poles along $Z$. By Saito's theory of mixed Hodge module (\cite{Sai88},\cite{Sai90}), this sheaf is a left $\mathscr{D}_X$-module which underlies the mixed Hodge module $j_{*}\mathbb{Q}_U^H[n]$, where $U$ is the complement of $Z$ and $j$ is the open embedding from $U$ to $X$. There are two filtrations on $\mathscr{O}_X(*Z)$, the Hodge filtration (denoted by $F_{\bullet}$) and the pole order filtration. Musta\c t\u a  and Popa \cite{MP16} introduce the notion of Hodge ideals to measure the difference of these two filtrations. More precisely, the Hodge ideals $\{I_k(Z)\}_{k\in \mathbb Z}$ of the divisor $Z$ are defined by
$$F_k\mathscr{O}_X(*Z)=I_k(Z)\otimes_{\mathscr{O}_X}\mathscr{O}_X\big((k+1)Z\big) \quad \text{for any } k \in \mathbb{Z}.$$
It turns out that $I_0(Z)=\mathcal{J}((1-\varepsilon)Z)$, the multiplier ideal of $(1-\varepsilon)Z$ where $\varepsilon$ is a sufficiently small positive number. Therefore, Hodge ideals can be viewed as generalizations of multiplier ideals.

In \cite{MP19a} and \cite{MP18}, the authors extend the notion of Hodge ideals to an arbitrary effective $\mathbb{Q}$-divisor $D$ on $X$. Locally we can assume that $D=\alpha\cdot \text{div}(h)$ for some nonzero regular function $h$ and some positive rational number $\alpha$. Denote by $Z$  the support of $D$. Instead of $\mathscr{O}_X(*Z)$ in the case of integral divisors, we consider the left $\mathscr{D}_X$-module $\mathscr{O}_X(*Z)h^{-\alpha}$, a rank 1 free $\mathscr{O}_X(*Z)$-module generated by the symbol $h^{-\alpha}$, on which the action of a derivation $P$ of $\mathscr{O}_X$ acts via the rule
$$P(wh^{-\alpha})=\left(P(w)-\alpha w\frac{P(h)}{h}\right)h^{-\alpha}.$$
Unlike $\mathscr{O}_X(*Z)$, this $\mathscr{D}_X$-module does not necessarily underlie a mixed Hodge module. However, there is a natural filtration on $\mathscr{O}_X(*Z)h^{-\alpha}$, denoted by $F_{\bullet}$, which makes it a filtered direct summand of some filtered $\mathscr{D}_X$-module underlying a mixed Hodge module. Similarly as the case of integral divisors, the Hodge ideals $\{I_k(D)\}_{k\in \mathbb{Z}}$ of $D$ are defined by the difference between the filtration $F_{\bullet}$ and the pole order filtration
$$F_k \big(\mathscr{O}_X(*Z)h^{-\alpha}\big)=I_k(D)\otimes_{\mathscr{O}_X} \mathscr{O}_X(kZ)h^{-\alpha}.$$
It's standard to check that the definition of Hodge ideals is independent of the choice of $h$ and $\alpha$. Therefore they can be defined globally on $X$.

However, $\mathscr{O}_X(*Z)h^{-\alpha}$ can not be defined globally since its definition depends on $\alpha$ and $h$. In order to obtain a global $\mathscr{D}_X$-module, we need a extra assumption: there is a line bundle $\mathcal{L}$  and an integer $l$ such that $l\cdot D$ is an integral divisor and 
$$\mathscr{O}_X(l\cdot D)\simeq \mathcal{L}^{\otimes l},$$
which is called the global assumption. With this assumption, in \cite[B.5]{MP19a} the authors construct a global filtered $\mathscr{D}_X$-module, denoted by $(\mathcal{M},F_{\bullet})$, which is locally filtered isomorphic to $\mathscr{O}_X(*Z)h^{-1/l}$. Moreover, $(\mathcal{M},F_{\bullet})$ underlies a direct summand of a mixed Hodge module. Therefore, by Saito's vanishing theorem \cite[\S 2.g]{Sai90}, given a ample line bundle $\mathcal{A}$ we have
\begin{align}\label{yong}
\mathbf{H}^{i}(X,\text{Gr}^F_k(\text{DR}(\mathcal{M}))\otimes_{\mathscr{O}_X} \mathcal{A})=0
\end{align}
for any integer $k$ and $i>0$, where 
$$\text{Gr}^F_k(\text{DR}(\mathcal{M}))=[\text{Gr}^F_k\mathcal{M}\rightarrow \Omega_X^1\otimes_{\mathscr{O}_X} \text{Gr}^F_{k+1}\mathcal{M}\rightarrow \cdots \rightarrow  \Omega_X^n\otimes_{\mathscr{O}_X} \text{Gr}^F_{k+n}\mathcal{M}]$$ 
(placed in degrees $-n,\cdots,0$) is the associated graded complex for the filtration on $\text{DR}(\mathcal{M})$ which is induced by the filtration $F_{\bullet}\mathcal{M}$ on $\mathcal{M}$. 

Under the global assumption, with the help of the vanishing (\ref{yong}), Musta\c t\u a  and Popa \cite{MP19a} show a vanishing theorem for Hodge ideals $\{I_k(D)\}_{k\in \mathbb{Z}}$ that in the case $k=0$ is nothing else but Nadel vanishing for multiplier ideals. However, the global assumption is not necessary for Nadel vanishing. They therefore conjecture that the vanishing still holds even without this extra hypothesis.

In this paper, we will remove the global assumption in  the vanishing theorem for Hodge ideals of $\mathbb{Q}$-divisors. This should be important for the sake of applications. Our main theorem is 

\begin{thm}[=Theorem \ref{van}]\label{thm1}
Let $X$ be a smooth projective complex variety of dimension $n$ and $D$ an effective $\mathbb{Q}$-divisor on $X$. Denote by $Z$ the support of $D$. For some $k\geq 0$, assume that $(X,D)$ is reduced $(k-1)$-log-canonical, i.e. $I_0(D)=\cdots=I_{k-1}(D)=\mathscr{O}_X(Z-\lceil D\rceil)$.

(1) Let $\mathcal{L}$ be a line bundle such that $\mathcal{L}-D$ is ample. If $\mathcal{L}((p+1)Z-\lceil D\rceil)$ is ample for all $0\leq p \leq k-1$, then
$$H^i(X,\omega_X \otimes \mathcal{L}(kZ)\otimes I_k(D))=0$$
for all $i \geq 2$. Moreover,
$$H^1(X,\omega_X \otimes \mathcal{L}(kZ)\otimes I_k(D))=0$$
holds if $$H^{k-p}\big(X,\Omega_X^{n-k+p}\otimes \mathcal{L}((p+1)Z-\lceil D\rceil)\big)=0$$ for all $0\leq p \leq k-1$.

(2) If there exists an ample effective divisor with support contained in $Z=\text{Supp}(D)$, then (1) also holds for $\mathcal{L}$ such that $\mathcal{L}-D$ is nef.
\end{thm}

The key points of our proof are as follows. Without the global assumption, there does not necessarily exist a global $\mathscr{D}_X$-module which is locally isomorphic to $\mathscr{O}_X(*Z)h^{-\alpha}$. However, fortunately, there exists a global sheaf
$$\mathcal{M}_k(D):=\frac{\omega_X(kZ)\otimes_{\mathscr{O}_X} I_k(D)}{\omega_X((k-1)Z)\otimes_{\mathscr{O}_X} I_{k-1}(D)}$$
which is locally isomorphic to $\omega_X\otimes\text{Gr}^F_k\big(\mathscr{O}_X(*Z)h^{-\alpha}\big)$ 
for any $k\in \mathbb{Z}$. And we introduce  a complex (see Definition \ref{def})
$$\text{Sp}_k(\mathcal{M}_{\bullet}(D))=[\mathcal{M}_k(D)\otimes\wedge^n T_X\rightarrow \mathcal{M}_{k+1}(D)\otimes\wedge^{n-1} T_X\rightarrow \cdots \mathcal{M}_{k+n}(D)\otimes\wedge^0 T_X]$$
placed in degrees $-n,\cdots, 0$, called the $k$-th Spencer complex of $\mathcal{M}_{\bullet}(D)$, which is locally isomorphic to the $k$-th graded piece of the de Rham complex of $\mathscr{O}_X(*Z)h^{-\alpha}:$$$[\text{Gr}^F_k\mathscr{O}_X(*Z)h^{-\alpha}\rightarrow \Omega_X^1\otimes \text{Gr}^F_{k+1}\mathscr{O}_X(*Z)h^{-\alpha}\rightarrow \cdots \rightarrow  \Omega_X^n\otimes \text{Gr}^F_{k+n}\mathscr{O}_X(*Z)h^{-\alpha}]$$
for any integer $k$.
Moreover, as an analogue to Saito's vanishing theorem for mixed Hodge modules (but without using it), a vanishing theorem for $\text{Sp}_k(\mathcal{M}_{\bullet}(D))$ is proved.  It plays a key role in our proof of vanishing for Hodge ideals of $\mathbb Q$-divisors.

\begin{thm}[=Theorem \ref{main}]\label{thm2}
Let $X$ be a smooth complex projective variety of dimension $n$ and $D$ an effective $\mathbb{Q}$-divisor. 

(1) If $\mathcal{A}$ is a line bundle on $X$ such that $\mathcal{A}-D$ is ample, then
$$\mathbf{H}^i(X,\text{Sp}_k(\mathcal{M}_{\bullet}(D))\otimes_{\mathscr{O}_X} \mathcal{A})=0$$
for any $i>0$ and any integer $k$.

(2) If there exists an ample effective divisor with support contained in $\text{Supp}(D)$, then (1) also holds for line bundle $\mathcal{A}$ such that $\mathcal{A}-D$ is nef.
\end{thm}

In fact, it's a direct consequence of the following two theorems.

\begin{thm}[=Theorem \ref{prop}]\label{thm3}
Let $D$ be an effective $\mathbb{Q}$-divisor on a smooth complex variety $X$ of dimension $n$. Let $f:Y\rightarrow X$ be a log resolution of $(X,D)$  which is isomorphic over $X\setminus D$. Denote by $ E$ the support of $f^*D$. Then 
$$\mathbf{R}f_*\big(\Omega_Y^k(\log  E)(-\lceil f^*D\rceil)\big)[n-k]$$ is quasi-isomorphic to $\text{Sp}_{-k}(\mathcal{M}_{\bullet}(D))$ for any integer $k$.
\end{thm}

\begin{thm}[=Theorem \ref{aa}]\label{thm}
Let $X$ be a smooth complex projective variety of dimension $n$ and $D$ an effective $\mathbb{Q}$-divisor. Let $f:Y\rightarrow X$ be a log resolution of $(X,D)$  which is isomorphic over $X\setminus D$. Denote by $ E$ the support of $f^*D$. 

(1) If $\mathcal{A}$ is a line bundle on $X$ such that 
$\mathcal{A}-D$ is ample, then
$$H^q\big(Y,\Omega^p_Y(\log  E)\otimes_{\mathscr{O}_Y} f^*\mathcal{A}(-\lceil f^*D\rceil)\big)=0$$
for any $p+q>n$.

(2) If there exists an ample effective divisor on $X$ with support contained in $\text{Supp}(D)$, then (1) also holds for line bundle $\mathcal{A}$ such that $\mathcal{A}-D$ is nef.
\end{thm}

Theorem \ref{thm} is a extension of Akizuki-Kodaira-Nakano vanishing theorem for log forms (see \cite{AN54} and \cite[Corollary 6.4]{EV92}), which says that given a simple normal crossing divisor $\Delta$ and an ample bundle $\mathcal{A}$ on a smooth variety $X$,
\begin{align}\label{no}
H^q(X,\Omega_X^p(\log \Delta)\otimes_{\mathscr{O}_X} \mathcal{A})=0
\end{align} 
for any $p+q>n$. It's showed in \cite[Theorem 32.2]{MP16} that (\ref{no}) also holds for semiample line bundles, provided that $X\setminus \Delta$ is affine.

In Section 2, we briefly review the definitions of Hodge ideals of $\mathbb Q$-divisors and how to describe Hodge ideals in terms of log resolutions.
In Section 3, we introduce the definition of $\mathcal{M}_k(D)$ and $\text{Sp}_k(\mathcal{M}_{\bullet}(D))$, and then we prove the  vanishing theorem for $\text{Sp}_k(\mathcal{M}_{\bullet}(D))$. Section 4 is dedicated to the proof of the vanishing theorem for Hodge ideals of $\mathbb Q$-divisors without the global assumption. 

At the end of this section we introduce some notations:

(1) If $\mathcal{A}$ and $\mathcal{B}$ are two $\mathscr{O}_X$-modules on a variety $X$, then $\mathcal{A}\otimes \mathcal{B}$ means the tensor product of $\mathcal{A}$ and $\mathcal{B}$ as $\mathscr{O}_X$-module unless otherwise stated (we omit the subscript $\mathscr{O}_X$).

(2) Let $D=\sum_i \alpha_i D_i$ be a $\mathbb{Q}$-divisor on a variety $X$,  then we denote 
$$\lceil D \rceil=\sum_i \lceil\alpha_i \rceil D_i \quad \text{and}\quad  \lfloor D \rfloor=\sum_i \lfloor\alpha_i \rfloor D_i$$
where $$\lceil\alpha_i \rceil =\min\{m\in \mathbb{Z}\mid m\geq \alpha_i\} \quad \text{ and }\quad \lfloor\alpha_i \rfloor =\max\{m\in \mathbb{Z}\mid m\leq \alpha_i\} $$
for all $i$.

(3) Let $\mathcal{E}$ be a vector bundle on a variety $X$, we denote by $S^k \mathcal{E}$ the $k$-th symmetric product of $\mathcal{E}$.

\section*{Acknowledgement}
I would like to express my sincere gratitude to Mihnea Popa for suggesting the problem and for his constant support of this project. I would like to thank Huaiqing Zuo for very helpful suggestions. I also thank Northwestern University for their warm hospitality during a visit when this paper is prepared. Lastly, I am greatly thankful to my advisor Stephen Shing-Toung Yau and Tsinghua Visiting Doctoral Students Foundation for providing financial support during my visit.

\section{Hodge ideal of $\mathbb{Q}$-divisor: birational approach}
In this section, we will review the definition of Hodge ideals of $\mathbb{Q}$-divisors and how to describe Hodge ideals by birational methods. For more details, please see Musta\c t\u a and Popa's paper \cite{MP19a}.

Let $X$ be a smooth complex variety of dimension $n$ and $D$ an effective $\mathbb{Q}$-divisor on $X$. Let $Z$ be the support of $D$.  Denote by $\mathscr{O}_X(*Z)$ the sheaf of rational functions on $X$ with poles along $Z$.

Locally, there exist a regular function $h$ and a rational number $\alpha$ such that $D=\alpha\cdot\text{div}(h)$. Consider the left $\mathscr{D}_X$-module $\mathscr{O}_X(*Z)h^{-\alpha}$, which is a rank 1 free $\mathscr{O}_X(*Z)$-module generated by the symbol $h^{-\alpha}$, and the action of a derivation $P$ of $\mathscr{O}_X$ is given by
$$P(wh^{-\alpha})=\left(P(w)-\alpha w\frac{P(h)}{h}\right)h^{-\alpha}.$$

There is a natural filtration on this $\mathscr{D}$-module, denoted by $F_{\bullet}\mathscr{O}_X(*Z)h^{-\alpha}$, which makes it a filtered direct summand of a filtered $\mathscr D_X$-module underlying a mixed Hodge module. Moreover, we can write 
$$F_k \big(\mathscr{O}_X(*Z)h^{-\alpha}\big)=I_k(D)\otimes \mathscr{O}_X(kZ)h^{-\alpha}$$
for an coherent ideal sheaf $I_k(D)$, which is called the $k$-th Hodge ideal of the $\mathbb Q$-divisor $D$. It's not hard to check that the definition of Hodge ideals is independent of the choice of $h$ and $\alpha$. Therefore, Hodge ideals can be defined globally on $X$ (see \cite[Section B.4]{MP19a}). Note that $F_k \big(\mathscr{O}_X(*Z)h^{-\alpha}\big)=0$ for $k<0$, which implies that $I_k(D)=0$ for $k<0$. 

The corresponding filtered right $\mathscr{D}_X$-module of $\mathscr{O}_X(*Z)h^{-\alpha}$ is $h^{-\alpha}\omega_X(*Z)$, on which the filtration is given by
$$F_{k-n} \big(h^{-\alpha}\omega_X(*Z)\big)=\omega_X\otimes F_k\big(\mathscr{O}_X(*Z)h^{-\alpha} \big).$$

$(h^{-\alpha}\omega_X(*Z),F_{\bullet})$ can be described in terms of log resolutions. Let $ f:Y \rightarrow X$ be a log resolution of the pair $(X,Z)$ which is isomorphic over $X\setminus Z$. Denote by $ E$ the support $ f^*Z$, then $ E$ is a reduced simple normal crossing divisor. Denote $g= f^*h$. It is showed in \cite[Section B.8]{MP19a} that there is a canonical filtered isomorphism 
\begin{align*}
\big(h^{-\alpha}\omega_X(*Z),F_{\bullet}\big)\simeq H^0f_+\big(g^{-\alpha}\omega_Y(* E),F_{\bullet}\big),
\end{align*}
and $$H^pf_+\big(g^{-\alpha}\omega_Y(* E),F_{\bullet}\big)=0 \quad \text{for } p\neq 0.$$
For the notation $f_+$ above, recall that for any proper morphism between smooth varieties $f:Y\rightarrow X$, there is a filtered direct image functor
$$f_+:\textbf{D}^b\big(\text{FM}(\mathscr{D}_Y)\big)\rightarrow \textbf{D}^b\big(\text{FM}(\mathscr{D}_Y)\big) $$
between the bounded derived categories of filtered $\mathscr{D}$-modules (See \cite{Sai88} and \cite[2.c]{Sai90}).
Moreover, Saito proves that this direct image functor is strict for filtered $\mathscr{D}_X$-modules underlying mixed Hodge modules. That is to say, the following diagram is commutative
\begin{align}\label{iso}
\xymatrix{
F_k\big(h^{-\alpha}\omega_X(*Z)\big)\ar@{^(->}[r] \ar[d]^{\simeq} \ar@{}[dr]|(.5){\circlearrowleft}&  h^{-\alpha}\omega_X(*Z)\ar[d]^{\simeq}\\
R^0f_*\big(F_k(g^{-\alpha}\omega_Y(* E)\overset{\mathbf{L}}{\otimes}_{\mathscr{D}_{Y}}\mathscr{D}_{Y\rightarrow X})\big) \ar@{^(->}[r] & R^0f_*(g^{-\alpha}\omega_Y(* E)\overset{\mathbf{L}}{\otimes}_{\mathscr{D}_{Y}}\mathscr{D}_{Y\rightarrow X})
 }
\end{align}
and  
\begin{equation}\label{iso2}
\begin{aligned}
&R^pf_*(g^{-\alpha}\omega_Y(* E)\overset{\mathbf{L}}{\otimes}_{\mathscr{D}_{Y}}\mathscr{D}_{Y\rightarrow X})\\
= &R^pf_*\big(F_k(g^{-\alpha}\omega_Y(* E)\overset{\mathbf{L}}{\otimes}_{\mathscr{D}_{Y}}\mathscr{D}_{Y\rightarrow X})\big)=0
\end{aligned}
\end{equation}
for any integer $k$ and any $p\neq 0$.
Here $\mathscr{D}_{Y\rightarrow X}=\mathscr{O}_Y\otimes_{f^{-1}\mathscr{O}_X}f^{-1}\mathscr{D}_X$ is the transfer $\mathscr{D}$-module.

Consider the following complex of locally free right $\mathscr{D}_Y$-modules
$$C^{\bullet}_{g^{-\alpha}}=[\mathscr{D}_Y\rightarrow \Omega_Y^1(\log  E)\otimes\mathscr{D}_Y\rightarrow \cdots \rightarrow \Omega_Y^n(\log  E)\otimes\mathscr{D}_Y]$$
placed in degrees $-n,\cdots,0$.
The differential $d$ is given by
$$d(\eta\otimes P)=d\eta\otimes P+\sum_i(dy_i\wedge \eta)\otimes \partial_{y_i}P-\alpha\frac{dg}{g} \wedge \eta \otimes P,$$
here $y_1,\cdots,y_n$ are local coordinates of $Y$.
In fact, $C^{\bullet}_{g^{-\alpha}}$ is a filtered complex, where
$$F_{k}C^{\bullet}_{g^{-\alpha}}=\Omega_Y^{\bullet+n}(\log  E)\otimes F_{k+n+\bullet}\mathscr{D}_Y.$$
It's not hard to check that the differential $C^i_{g^{-\alpha}}\rightarrow C^{i+1}_{g^{-\alpha}}$ induces a map 
\begin{align*}
C^i_{g^{-\alpha}}(-\lceil f^*D\rceil):=\mathscr{O}_Y(-\lceil f^*D\rceil)\otimes C^i_{g^{-\alpha}} \\
\rightarrow C^{i+1}_{g^{-\alpha}}(-\lceil f^*D\rceil):=\mathscr{O}_Y(-\lceil f^*D\rceil)\otimes C^{i+1}_{g^{-\alpha}}.
\end{align*}
We thus obtain a filtered subcomplex $(C^{\bullet}_{g^{-\alpha}}(-\lceil f^*D\rceil),F_{\bullet})$ of $(C^{\bullet}_{g^{-\alpha}},F_{\bullet})$. 
Note that this is not obtained by just tensoring $C^{\bullet}_{g^{-\alpha}}$ with $\mathscr{O}_Y(-\lceil f^*D\rceil)$, since $C^{\bullet}_{g^{-\alpha}}$ is not a complex of $\mathscr{O}_Y$-modules. 

There is a natural morphism 
\begin{align*}
C^0_{g^{-\alpha}}(-\lceil f^*D\rceil)\otimes \mathscr{D}_X=\omega_Y(E-\lceil f^*D\rceil)\otimes\mathscr{D}_X &\longrightarrow g^{-\alpha}\omega_Y(* E)\\
w\otimes Q &\longmapsto (g^{-\alpha}w)\cdot Q
\end{align*}
\cite[Proposition 6.1]{MP19a} tells us that this morphism make the filtered complex $(C^{\bullet}_{g^{-\alpha}}(-\lceil f^*D\rceil),F_{\bullet})$ a filtered resolution of $(g^{-\alpha}\omega_Y(* E),F_{\bullet})$. Using (\ref{iso}), (\ref{iso2}) and this result, Musta\c t\u a and Popa prove the following theorem:

\begin{thm}\cite[Theorem 8.1]{MP19a}
With the above notation, the followings hold:

i) For every $p\neq 0$ and every $k\in \mathbb{Z}$, we have 
$$R^pf_*(C^{\bullet}_{g^{-\alpha}}(-\lceil f^*D\rceil)\otimes_{\mathscr{D}_Y}\mathscr{D}_{Y \rightarrow X})=0$$
and
$$R^pf_*F_k(C^{\bullet}_{g^{-\alpha}}(-\lceil f^*D\rceil)\otimes_{\mathscr{D}_Y}\mathscr{D}_{Y \rightarrow X})=0,$$

ii) For every $k\in \mathbb{Z}$, the natural inclusion induces an injective map
$$R^0f_*F_k(C^{\bullet}_{g^{-\alpha}}(-\lceil f^*D\rceil)\otimes_{\mathscr{D}_Y}\mathscr{D}_{Y \rightarrow X}) \hookrightarrow R^0f_*(C^{\bullet}_{g^{-\alpha}}(-\lceil f^*D\rceil)\otimes_{\mathscr{D}_Y}\mathscr{D}_{Y \rightarrow X})$$

iii) We have a canonical isomorphism 
$$R^0f_*(C^{\bullet}_{g^{-\alpha}}(-\lceil f^*D\rceil)\otimes_{\mathscr{D}_Y}\mathscr{D}_{Y \rightarrow X})\simeq h^{-\alpha}\omega_X(*Z)$$
that induces for every $k\in \mathbb{Z}$ an isomorphism
\begin{align*}
R^0f_*F_{k-n}(C^{\bullet}_{g^{-\alpha}}(-\lceil f^*D\rceil)\otimes_{\mathscr{D}_Y}\mathscr{D}_{Y \rightarrow X})&\simeq F_{k-n} \big(h^{-\alpha}\omega_X(*Z)\big)\\
&\simeq h^{-\alpha}\omega_X(kZ)\otimes I_k(D).
\end{align*}
\end{thm}
Denote
$$\text{Gr}^F_k(C^{\bullet}_{g^{-\alpha}}(-\lceil f^*D\rceil)\otimes_{\mathscr{D}_Y}\mathscr{D}_{Y \rightarrow X})=\frac{F_k(C^{\bullet}_{g^{-\alpha}}(-\lceil f^*D\rceil)\otimes_{\mathscr{D}_Y}\mathscr{D}_{Y \rightarrow X})}{F_{k-1}(C^{\bullet}_{g^{-\alpha}}(-\lceil f^*D\rceil)\otimes_{\mathscr{D}_Y}\mathscr{D}_{Y \rightarrow X})}$$
for any integer $k$.
The following corollary is a direct consequence of the above theorem.
\begin{cor}\label{cor1}
i) For every $p \neq 0$ and every $k\in \mathbb{Z}$, we have
$$R^pf_*\text{Gr}^F_k(C^{\bullet}_{g^{-\alpha}}(-\lceil f^*D\rceil)\otimes_{\mathscr{D}_Y}\mathscr{D}_{Y \rightarrow X})=0$$
ii) For any $k \in \mathbb{Z}$, there is a canonical isomorphism
\begin{align*}
R^0f_*\text{Gr}^F_{k-n}(C^{\bullet}_{g^{-\alpha}}(-\lceil f^*D\rceil)\otimes_{\mathscr{D}_Y}\mathscr{D}_{Y \rightarrow X})&\simeq \text{Gr}^F_{k-n}\big(h^{-\alpha}\omega_X(*Z)\big)\\
&\simeq \omega_X \otimes\text{Gr}_k^F(\mathscr{O}_X(*Z)h^{-\alpha}).
\end{align*}
\end{cor}

By definition, $\text{Gr}_k^F(C^{\bullet}_{g^{-\alpha}}\otimes_{\mathscr{D}_Y} \mathscr{D}_{Y\rightarrow X})$ is the complex 
\begin{equation}\label{eq1}
\begin{aligned}
0 \rightarrow \mathscr{O}_Y\otimes  f^*S^{k}T_X \rightarrow \Omega_Y^1(\log  E)&\otimes f^*S^{k+1}T_X \rightarrow 
\cdots \\
&\rightarrow \Omega_Y^n(\log  E)\otimes f^*S^{k+n}T_X\rightarrow 0
\end{aligned}
\end{equation}
placed in degrees $-n,\cdots,0$, where $S^iT_X$ means the $i$-th symmetric power of $T_X$. And the differential is given by
\begin{equation}\label{eq2}
d(\eta\otimes f^*P)=\sum_i(dy_i\wedge \eta)\otimes \partial_{y_i}(f^*P),
\end{equation}
where $y_1,\cdots,y_n$ are local coordinates on $Y$.
Since the differential is $\mathscr{O}_Y$-linear, $\text{Gr}_k^F(C^{\bullet}_{g^{-\alpha}}\otimes_{\mathscr{D}_Y} \mathscr{D}_{Y\rightarrow X})$ is a complex of $\mathscr{O}_Y$-modules. Therefore, the complex 
$$\text{Gr}_k^F(C^{\bullet}_{g^{-\alpha}}(-\lceil f^*D\rceil)\otimes_{\mathscr{D}_Y} \mathscr{D}_{Y\rightarrow X})$$
 is obtained by tensoring $\text{Gr}_k^F(C^{\bullet}_{g^{-\alpha}}\otimes_{\mathscr{D}_Y} \mathscr{D}_{Y\rightarrow X})$ with $\mathscr{O}_Y(-\lceil f^*D\rceil)$.

Since $h^{-\alpha}\omega_X(*Z)$ is a filtered right $\mathscr{D}_X$-module, $\bigoplus_k \text{Gr}_k^F(h^{-\alpha}\omega_X(*Z))$ has a right $\bigoplus_k S^kT_X$-module structure, which induces a morphism
\begin{equation}\label{map}
\begin{aligned}
d:\text{Gr}_k^F(h^{-\alpha}\omega_X(*Z))\otimes  \wedge^q T_X&\longrightarrow \text{Gr}_{k+1}^F(h^{-\alpha}\omega_X(*Z))\otimes  \wedge^{q-1} T_X\\
w\otimes\theta_1\wedge\cdots \wedge\theta_q &\longmapsto  \sum_{i} (-1)^{i-1} w\cdot \theta_i\otimes (\theta_1\wedge\cdots\hat{\theta_i}\cdots \wedge \theta_q)
\end{aligned}
\end{equation}
for any integer $q$ and $k$. It's easy to check that $d\circ d=0$, hence we obtain a complex
\begin{align*}
0\rightarrow \text{Gr}_{k-n}^F(h^{-\alpha}\omega_X(*Z))\otimes  \wedge^n T_X\rightarrow \text{Gr}_{k-n+1}^F(h^{-\alpha}\omega_X(*Z))\otimes  \wedge^{n-1} T_X \\ \cdots \rightarrow \text{Gr}_{k}^F(h^{-\alpha}\omega_X(*Z))\otimes  \wedge^0 T_X\rightarrow 0
\end{align*}
which is the $k$-th graded piece of the Spencer complex of $h^{-\alpha}\omega_X(*Z)$, denoted by $\text{Gr}_k^F\text{Sp}(h^{-\alpha}\omega_X(*Z))$. Note that this is isomorphic to the $k$-th graded piece of the de Rham complex of $\mathscr{O}_X(*Z)h^{-\alpha}$
\begin{align*}
0\rightarrow \text{Gr}^F_k(\mathscr{O}_X(*Z)h^{-\alpha})\rightarrow \Omega_X^1\otimes &\text{Gr}^F_{k+1}(\mathscr{O}_X(*Z)h^{-\alpha})\rightarrow \\
& \cdots \rightarrow  \Omega_X^n\otimes \text{Gr}^F_{k+n}(\mathscr{O}_X(*Z)h^{-\alpha})\rightarrow 0.
\end{align*}
Since 
\begin{align*}
h^{-\alpha}\omega_X(*Z)&\simeq H^0 f_+(g^{-\alpha}\omega_Y(* E))\\
&\simeq R^0f_*(C^{\bullet}_{g^{-\alpha}}(-\lceil f^*D\rceil)\otimes_{\mathscr{D}_Y}\mathscr{D}_{Y \rightarrow X}),
\end{align*}
the right $\mathscr{D}_X$-module structure of $h^{-\alpha}\omega_X(*Z)$ is induced by the right $f^{-1}\mathscr{D}_X$-module structure of $C^{\bullet}_{g^{-\alpha}}(-\lceil f^*D\rceil)\otimes_{\mathscr{D}_Y}\mathscr{D}_{Y \rightarrow X}$. Therefore, the right $\bigoplus_k S^kT_X$-module structure of $\bigoplus_k \text{Gr}_k^F(h^{-\alpha}\omega_X(*Z))$ is induced by the right $f^{-1}(\bigoplus_k S^kT_X)$-module structure of $\bigoplus_k \text{Gr}^F_{k-n}(C^{\bullet}_{g^{-\alpha}}(-\lceil f^*D\rceil)\otimes_{\mathscr{D}_Y}\mathscr{D}_{Y \rightarrow X})$. That is to say, the morphism (\ref{map}) is obtained by applying the functor $R^0f_*$ to the morphism between the follwing two complexes

\begin{align*}
\text{Gr}^F_{k}(C^{\bullet}_{g^{-\alpha}}(-\lceil f^*D\rceil)\otimes_{\mathscr{D}_Y}\mathscr{D}_{Y \rightarrow X})\otimes f^*(\wedge^q T_X):&\\
[\mathscr{O}_Y(-\lceil  f^*D\rceil)\otimes  f^*(S^{k}T_X\otimes \wedge^q T_X) \rightarrow \\ 
\cdots  \rightarrow \Omega_Y^n(\log  E)&(-\lceil  f^*D\rceil)\otimes f^*(S^{k+n}T_X\otimes \wedge^{q} T_X)]
\end{align*}
and 
\begin{align*}
\text{Gr}^F_{k+1}(C^{\bullet}_{g^{-\alpha}}(-\lceil f^*D\rceil)\otimes_{\mathscr{D}_Y}\mathscr{D}_{Y \rightarrow X})\otimes f^*(\wedge^{q-1} T_X):&\\
[\mathscr{O}_Y(-\lceil  f^*D\rceil)\otimes  f^*(S^{k+1}T_X\otimes \wedge^{q-1} T_X) \rightarrow \\ \cdots 
\rightarrow \Omega_Y^n(\log  E)(-\lceil  f^*D\rceil)\otimes &f^*(S^{k+n+1}T_X\otimes \wedge^{q-1} T_X)]
\end{align*}
which is induced by the natural morphism
\begin{align*}
S^p T_X\otimes \wedge^q T_X& \longrightarrow S^{p+1} T_X \otimes\wedge^{q-1} T_X \\
P\otimes(\theta_1\wedge\cdots \wedge\theta_q)&\longmapsto\sum_{i} (-1)^{i-1} P\cdot \theta_i\otimes (\theta_1\wedge\cdots\hat{\theta_i}\cdots \wedge \theta_q).
\end{align*}

\section{GLobal approach}
In the previous section, using the local defining function $h$ of $1/\alpha\cdot D$, we introduce the $\mathscr{D}_X$-module $h^{-\alpha}\omega_{X}(*Z)$ and  the complex $C^{\bullet}_{g^{-\alpha}}$ to calculate Hodge ideals. However, since $h$ is not a global function, both $\mathscr{O}_{X}(*Z)h^{-\alpha}$ and $C^{\bullet}_{g^{-\alpha}}$ can not be defined globally.

In this section, we will treat the problem globally. Let $X$ be a smooth complex variety of dimensional $n$ and $D$ an effective $\mathbb{Q}$-divisor on $X$. Denote by $Z$ the support of $D$. Fix a positive rational number $\alpha$ such that $1/\alpha\cdot D$ is an  integral divisor. Let $ f:Y \rightarrow X$ be a log resolution of the pair $(X,Z)$ which is isomorphic over $X\setminus Z$. Denote by $ E$ the support of $ f^*Z$.

\begin{defn}\label{def1} Define a complex on $Y$
\begin{equation}\label{dG}
\begin{aligned}
G_k^{\bullet}:=
&[\mathscr{O}_Y\otimes  f^*S^{k}T_X\rightarrow \Omega_Y^1(\log  E)\otimes f^*S^{k+1}T_X\rightarrow \\
&\cdots \rightarrow \Omega_Y^n(\log  E)\otimes f^*S^{k+n}T_X].
\end{aligned}
\end{equation}
placed in degrees $-n,\cdots,0$.
The differential is given by
\begin{align}\label{G}
d(\eta\otimes P)=\sum_i(dy_i\wedge \eta)\otimes \partial_{y_i}P
\end{align}
where $y_1,\cdots,y_n$ are local coordinates of $Y$.
\end{defn}

The differential of $G_k^{\bullet}$ is $\mathscr{O}_Y$-linear, that is to say, $G_k^{\bullet}$ is a complex of $\mathscr{O}_Y$-modules. Denote  the subcomplex  $G_k^{\bullet}\otimes\mathscr{O}_Y(-\lceil f^*D \rceil)$ of $G_k^{\bullet}$ by $G_k^{\bullet}(-\lceil f^*D \rceil)$. Define
$$\mathcal{M}_k(D):=\frac{\omega_X(kZ)\otimes I_k(D)}{\omega_X((k-1)Z)\otimes I_{k-1}(D)}.$$
for any integer $k$. If $D=\alpha \cdot\text{div}(h)$ for some regular function $h$, then
$$\mathcal{M}_k(D)\simeq \text{Gr}_{k-n}^F(h^{-\alpha}\omega_X(*Z))\simeq \omega_X\otimes\text{Gr}^F_k\big(\mathcal{O}_X(*Z)h^{-\alpha}\big)$$
as $\mathscr{O}_X$-modules.
\begin{prop}\label{pp}
For any $k\in \mathbb{Z}$,

i) $R^p f_*(G_k^{\bullet}(-\lceil f^*D\rceil))=0$ for any $p\neq 0$;

ii) $R^0 f_*(G_k^{\bullet}(-\lceil f^*D\rceil)) \simeq \mathcal{M}_{k+n}(D)$.
\end{prop}
\begin{proof}
For short we denote $f^*D$ by $G$. Locally, we may suppose that there exists a regular function $h$ such that $1/\alpha\cdot D=\text{div}(h)$. Denote $g=f^*h$. Comparing Definition \ref{def1} with (\ref{eq1}) and (\ref{eq2}) we see that locally
$$G_k^{\bullet}(-\lceil G\rceil)\simeq\text{Gr}_k^F(C^{\bullet}_{g^{-\alpha}}(-\lceil G\rceil)\otimes_{\mathscr{D}_Y} \mathscr{D}_{Y\rightarrow X}).$$
Then Corollary \ref{cor1}(i) implies the first assertion in the proposition.

For $p=0$, by Corollary \ref{cor1}(ii) we see that locally
\begin{align*}
R^0f_*(G_k^{\bullet}(-\lceil G\rceil))
&\simeq \text{Gr}_k^F(h^{-\alpha}\omega_X(*Z))\\
&\simeq \mathcal{M}_{k+n}(D).
\end{align*}
In order to obtain a global isomorphism, it suffices to the check that the above local isomorphism is independent of the choice of the local defining function $h$ of $1/\alpha \cdot D$. Suppose that  $1/\alpha\cdot D=\text{div}(h_1)=\text{div}(h_2)$ for two regular functions $h_1$ and $h_2$ on a open subset $U$ of $X$. Write $\alpha=i/l$ for two positive integers $i$ and $l$. Shrinking the open set $U$, we may assume that there exists an invertible regular function $u$ on $U$ such that $h_1=u^lh_2$. Denote $g_1=f^*h_1$ and $g_2=f^*h_2$. Then there are two natural isomorphisms
\begin{align*}
\phi:h_1^{-\alpha}\omega_X(*Z)&\longrightarrow h_2^{-\alpha}\omega_X(*Z)\\
h_1^{-\alpha}w&\longmapsto h_2^{-\alpha}u^{-i}w
\end{align*}
and 
\begin{align*}
\widetilde{\phi}: g_1^{-\alpha}\omega_Y(* E)&\longrightarrow g_2^{-\alpha}\omega_Y(* E)\\
g_1^{-\alpha}w&\longmapsto g_2^{-\alpha}(f^*u)^{-i}w
\end{align*}
such that the following diagram is commutative 
$$
\xymatrix{
    f_+\left(g_1^{-\alpha}\omega_Y(* E)\right)\ar@{}[dr]|(.5){\circlearrowleft}\ar[d]^{\simeq}\ar[r]^{f_+\widetilde{\phi}} & f_+\left(g_2^{-\alpha}\omega_Y(* E)\right) \ar[d]^{\simeq}\\
 h_1^{-\alpha}\omega_X(*Z) \ar[r]^{\phi}& h_2^{-\alpha}\omega_X(*Z).
}
$$
Therefore we have the following commutative diagram
$$
\xymatrix{
R^0f_*\big(\text{Gr}_k^F(C^{\bullet}_{g_1^{-\alpha}}(-\lceil G\rceil)\otimes_{\mathscr{D}_Y} \mathscr{D}_{Y\rightarrow X})\big)\ar@{}[dr]|(.5){\circlearrowleft}\ar[d]^{\simeq}\ar[r] & R^0f_*\big(\text{Gr}_k^F(C^{\bullet}_{g_2^{-\alpha}}(-\lceil G\rceil)\otimes_{\mathscr{D}_Y} \mathscr{D}_{Y\rightarrow X})\big) \ar[d]^{\simeq}\\
\text{Gr}_k^F(h_1^{-\alpha}\omega_X(*Z))  \ar[r] & \text{Gr}_k^F(h_2^{-\alpha}\omega_X(*Z))\\
}
$$
It induces
$$
\xymatrix{
R^0f_*(G_{k}^{\bullet}(-\lceil G\rceil))\ar@{}[dr]|(.5){\circlearrowleft}\ar[d]_{\psi_1}^{\simeq}\ar[r]^{\widetilde{\varphi}} & R^0f_*(G_{k}^{\bullet}(-\lceil G\rceil)) \ar[d]_{\psi_2}^{\simeq}\\
\mathcal{M}_{k+n}(D)\ar[r]^{\varphi}& \mathcal{M}_{k+n}(D).
}
$$
where both $\widetilde{\varphi}$ and $\varphi$ are given by multiplication by $u^{-i}$. It suffices to check that $\psi_1=\psi_2$. For any local section $s$ of $R^0f_*(G_{k}^{\bullet}(-\lceil G\rceil))$, denote $\psi_1(s)$ by $t$. Then
$$\varphi\circ \psi_1(s)=\varphi(t)=u^{-i}t.$$
On the other hand, 
$$\psi_2\circ \widetilde{\varphi}(s)=\psi_2(u^{-i}s).$$
Hence $\psi_2(u^{-i}s)=u^{-i}t$, which implies that $\psi_2(s)=t$. Therefore $\psi_1=\psi_2$. 

\end{proof}

For any integers $p\geq 1$ and $q\geq 0$, there is a natural morphism of $\mathscr{O}_X$-modules
\begin{align}\label{diff}
d: S^p T_X\otimes \wedge^q T_X  \longrightarrow S^{p+1} T_X \otimes\wedge^{q-1} T_X 
\end{align}
which is given by
$$d(P\otimes\theta_1\wedge\cdots \wedge\theta_q)=\sum_{i} (-1)^{i-1} P\cdot \theta_i\otimes (\theta_1\wedge\cdots\hat{\theta_i}\cdots \wedge \theta_q).$$
It's easy to check that 
\begin{align}\label{use}
d\circ d=0.
\end{align}
 So we obtain a complex
\begin{equation}\label{dH}
\begin{aligned}
H^{\bullet}_k:\quad 0\rightarrow  S^kT_X \otimes \wedge^nT_X &\stackrel{d}{\longrightarrow} S^{k+1}T_X \otimes \wedge^{n-1}T_X\rightarrow \cdots \\ 
&\stackrel{d}{\longrightarrow} S^{k+n-1}T_X \otimes \wedge^1T_X 
\stackrel{d}{\longrightarrow} S^{k+n}T_X\rightarrow 0.
\end{aligned}
\end{equation}
placed in degrees $-n,\cdots,0$. Note that $S^iT_X=0$ for $i<0$. Then
 $$H^{\bullet}_k=0\quad \text{for } k<-n$$
 and $$H^{\bullet}_{-n}=S^0T_X=\mathscr{O}_X.$$
For $k>-n$, we have the following classical result.

\begin{lem}\cite[Example B.2.1]{Laz04} \label{l}
For any $k\geq -n+1$, the complex $H^{\bullet}_k$ is exact.
\end{lem}
It's easy to check the following diagram is commutative for any $i,p,q$
$$
\xymatrix{
\Omega_Y^i(\log  E)\otimes f^*(S^{p}T_X\otimes \wedge^qT_X)\ar@{}[rd]|(0.5){\circlearrowright}\ar[r]\ar[d] &\Omega_Y^{i+1}(\log  E)\otimes f^*(S^{p+1}T_X\otimes \wedge^qT_X)\ar[d]\\
\Omega_Y^{i+1}(\log  E)\otimes f^*(S^{p+1}T_X\otimes \wedge^{q-1}T_X)\ar[r]&\Omega_Y^i(\log  E)\otimes f^*(S^{p+2}T_X\otimes \wedge^{q-1}T_X)}
$$
where the vertical morphisms are induced by (\ref{diff}) and the horizontal morphisms are induced by the differential of $G_k^{\bullet}$ (see the expression (\ref{G})).
Therefore the morphism (\ref{diff}) induces a morphism between complexes of $\mathscr{O}_Y$-modules
\begin{align*}
G_k^{\bullet}(-\lceil f^*D\rceil)\otimes f^*(\wedge^qT_X) : \quad &[\mathscr{O}_Y(-\lceil f^*D\rceil)\otimes  f^*(S^{k}T_X\otimes \wedge^qT_X)\rightarrow\cdots\\
& \rightarrow \Omega_Y^n(\log  E)(-\lceil f^*D\rceil)\otimes f^*(S^{k+n}T_X\otimes \wedge^qT_X)]
\end{align*}
and  
\begin{align*}
G_{k+1}^{\bullet}(-\lceil f^*D\rceil)\otimes f^*(\wedge^{q-1}T_X): \quad [\mathscr{O}_Y(-\lceil f^*D\rceil)\otimes  f^*(S^{k+1}T_X\otimes \wedge^{q-1}T_X)\rightarrow\cdots &\\
 \rightarrow \Omega_Y^n(\log  E)(-\lceil f^*D\rceil)\otimes f^*(S^{k+n+1}T_X\otimes \wedge^{q-1}T_X)]&.
\end{align*}
Applying the functor $R^0f_*$ to the above morphism, by Proposition \ref{pp}(ii) we obtain a morphism 
$$\widetilde{d}: \mathcal{M}_{k+n}(D)\otimes \wedge^q T_X \rightarrow \mathcal{M}_{k+n+1}(D)\otimes \wedge^{q-1}T_X.$$
From (\ref{use}) we see that $\widetilde{d}\circ \widetilde{d}=0$.

\begin{defn}\label{def}
Define $\text{Sp}_k(\mathcal{M}_{\bullet}(D))$ to be the following complex of $\mathscr{O}_X$-modules
$$[\mathcal{M}_k(D)\otimes\wedge^n T_X\stackrel{\widetilde{d}}{\longrightarrow} \mathcal{M}_{k+1}(D)\otimes\wedge^{n-1} T_X\stackrel{\widetilde{d}}{\longrightarrow} \cdots \stackrel{\widetilde{d}}{\longrightarrow}\mathcal{M}_{k+n}(D)\otimes\wedge^0 T_X]$$
placed in degrees $-n,-n+1,\cdots,0$. 
\end{defn}
\begin{rem}\label{rem}
Locally we can suppose that there exists a regular function $h$ such that $D=\alpha \cdot \text{div}(h)$. Then by definition, the complex $\text{Sp}_k(\mathcal{M}_{\bullet}(D))$ is locally isomorphic to $\text{Gr}^F_k\text{Sp}(h^{-\alpha}\omega_X(*Z))$. Since the definition of  $\text{Gr}^F_k\text{Sp}(h^{-\alpha}\omega_X(*Z))$ is intrinsic, we see that $\text{Sp}_k(\mathcal{M}_{\bullet}(D))$ is also independent of the choice of log resolutions.
 
\end{rem}

\begin{thm}[=Theorem \ref{thm3}]\label{prop}
For any integer $k$, we have 
$$\mathbf{R}f_*\big(\Omega_Y^k(\log  E)(-\lceil f^*D\rceil)\big)[n-k]$$
is quasi-isomorphic to $\text{Sp}_{-k}(\mathcal{M}_{\bullet}(D)).$
\end{thm}
\begin{proof}
Consider the following double complex
%$$
%\xymatrix{
%   \widetilde{\Omega}^n_Y\otimes f^*(\wedge^{n-k}T_X) \ar[r] & \widetilde{\Omega}^n_Y\otimes f^*(S^1T_X\otimes \wedge^{n-k-1}T_X)\ar[r]&\cdots\ar[r] &\widetilde{\Omega}^n_Y\otimes f^*S^{n-k}T_X\\
%  & \widetilde{\Omega}^{n-1}_Y\otimes f^*(\wedge^{n-k-1}T_X)\ar[u]\ar[r]&\cdots\ar[r] &\widetilde{\Omega}^{n-1}_Y\otimes f^*S^{n-k-1}T_X \ar[u]\\
%  && \widetilde{\Omega}^{k+1}_Y\otimes f^*(\wedge^1T_X) \ar[r]\ar@{}[ul]|(.5){\cdots}& \widetilde{\Omega}^{k+1}_Y \otimes f^*S^1T_X\ar@{}[u]|(.5){\cdots}\\
%  &&&\widetilde{\Omega}^k_Y\ar[u]
%}.
%$$
$$C^{i,j}=\Omega_Y^{n+i}(\log  E)(-\lceil f^*D\rceil)\otimes f^*(S^{n-k+i+j}T_X\otimes \wedge^{-j} T_X)$$ 
where the horizontal differential is induced by (\ref{G}) and the vertical differential is induced by (\ref{diff}). That is to say,
\begin{align}\label{e1}
C^{\bullet,j}=G_{-k+j}^{\bullet}(-\lceil f^*D\rceil)\otimes f^*(\wedge^{-j}T_X)
\end{align}
and 
\begin{align}\label{e2}
C^{i,\bullet}=\Omega_Y^{n+i}(\log  E)(-\lceil f^*D\rceil)\otimes f^*(H^{\bullet}_{-k+i}).
\end{align}
For the expression of $G_{-k+j}^{\bullet}$ (resp. $H^{\bullet}_{-k+i})$ please see (\ref{dG}) (resp.(\ref{dH})).
By Lemma \ref{l}, we know that $C^{i,\bullet}$ is acyclic for $i>k-n$.
It's easy to see that 
$$C^{i,\bullet}=0 \quad \text{for }i<k-n$$
and  $$C^{k-n,\bullet}=\Omega_Y^{k}(\log  E)(-\lceil f^*D\rceil).$$ 
As a result, the single complex associated with the above double complex (denoted by $C^{\bullet}$) is quasi-isomorphic to $\Omega_Y^{k}(\log  E)(-\lceil f^*D\rceil)[n-k]$.  It suffices to prove that $\mathbf{R}f_*C^{\bullet}$ is quasi-isomorphic to $\text{Sp}_{-k}(\mathcal{M}_{\bullet}(D))$.

There is a spectral sequence
$$E_1^{i,j}=R^i f_*(C^{\bullet,j})\Longrightarrow R^{i+j} f_*(C^{\bullet}).$$
By (\ref{e1}) and Proposition \ref{pp}, we have
$$E_1^{0,j}\simeq \mathcal{M}_{n-k+j}(D)\otimes \wedge^{-j}T_X$$ 
and 
$$E_1^{i,j}=0 \quad \text{for } i\neq 0.$$ Therefore, the spectral sequence degenerates at $E_2$. By definition, the following diagram is commutative,
$$
\xymatrix{
E_1^{0,j}\ar[d]_{\simeq}\ar@{}[rd]|(.5){\circlearrowleft}\ar[r]^{d_1^{0,j}}& E_1^{0,j+1}\ar[d]_{\simeq}\\ \mathcal{M}_{n-k+j}(D)\otimes \wedge^{-j}T_X\ar[r]^{\widetilde{d}}& \mathcal{M}_{n-k+j+1}(D)\otimes \wedge^{-j-1}T_X
}
$$
where $\widetilde{d}$ is the differential of $\text{Sp}_{-k}(\mathcal{M}_{\bullet}(D))$.
It follows that 
$$E_2^{0,j}\simeq H^j(\text{Sp}_{-k}(\mathcal{M}_{\bullet}(D)) \quad \text{for any } j.$$
As a result, $\mathbf{R}f_*C^{\bullet}$ is quasi-isomorphic to $\text{Sp}_{-k}(\mathcal{M}_{\bullet}(D))$.
\end{proof}
\begin{rem}
Theorem \ref{prop} is consistent with some known local Nakano-type vanishing results for $\mathbb{Q}$-divisors. 

(1) For any $p+q>n$,
$$R^pf_*\big(\Omega_Y^q(\log  E)(-\lceil f^*D\rceil)\big)=0.$$ 
This result is first obtained by Saito \cite[Corollary 3]{Sai07} when $D$ is reduced and the general result is obtained in \cite[Corollary C]{MP19a}. 

(2) If $D=\alpha \cdot\text{div}(h)$ for some global regular function $h$ and some rational number $\alpha$,  then 
$$R^pf_*\big(\Omega_Y^{n-p}(\log  E)(-\lceil f^*D\rceil)\big)=0 \quad \text{for all } p>k$$
if and only if the Hodge filtration on $\mathscr{O}_X(*Z)h^{-\alpha}$ is generated at level $k$, that is to say, the natural morphism $$F_1\mathscr{D}_X\cdot F_p\mathscr{O}_X(*Z)h^{-\alpha}\rightarrow F_{p+1}\mathscr{O}_X(*Z)h^{-\alpha}$$ is surjective for any $p\geq k$. This result is obtained in \cite[Theorem 17.1]{MP16} when $D$ is reduced and the general statement is given in \cite[Theorem 10.1]{MP19a}. 

Moreover, it's proved in \cite[Theorem E]{MP19b} that if $\text{div}(h)$ is reduced and $0<\alpha\leq 1$, then the Hodge filtration on $\mathscr{O}_X(*Z)h^{-\alpha}$ is generated at level $n-\lceil \widetilde{\alpha}_{h}+\alpha\rceil$, where $\widetilde{\alpha}_{h}$ is the negative of the largest root of the reduced Bernstein-Sato polynomial $\widetilde{b}_h(s)=b_h(s)/(s+1)$. As a result, 
\begin{thm}\cite[Theorem F]{MP19b}\label{3.10}
If $h$ defines a reduced divisor and $\alpha$ is a rational number in $(0,1]$, then 
$$R^pf_*\big(\Omega_Y^{n-p}(\log  E)(-\lceil f^*D\rceil)\big)=0 \quad \text{for all } p>n-\lceil \widetilde{\alpha}_{h}+\alpha\rceil.$$
\end{thm}
\begin{rem}
As is mentioned in \cite[Remark 5.5]{MP19b}, the same statements holds if $\text{div}(h)$ is not reduced but $\lceil D \rceil=Z$. The condition $\lceil D \rceil=Z$ is needed just  because \cite[Theorem $A$]{MP18} requires it. However, in fact, in the proof of \cite[Theorem $A$]{MP18}, this condition has not been used. Therefore, the assumption that $\text{div}(h)$ is reduced can be removed in Theorem \ref{3.10}.
\end{rem}
\end{rem}

Now we can prove the vanishing theorem for $\text{Sp}_k(\mathcal{M}_{\bullet}(D))$.
\begin{thm}[=Theorem \ref{thm2}]\label{main}
Let $X$ be a smooth complex projective variety of dimension $n$ and $D$ an effective $\mathbb{Q}$-divisor. 

(1) If $\mathcal{A}$ is a line bundle on $X$ such that $\mathcal{A}-D$ is ample, then
$$\mathbf{H}^i(X,\text{Sp}_k(\mathcal{M}_{\bullet}(D))\otimes \mathcal{A})=0$$
for any $i>0$ and any integer $k$.

(2) If there exists an ample effective divisor with support contained in $\text{Supp}(D)$, then (1) also holds for line bundle $\mathcal{A}$ such that $\mathcal{A}-D$ is nef.
\end{thm}
\begin{proof}
The assertions follow from Theorem \ref{prop} and the following Theorem.
\end{proof}
\begin{thm}[=Theorem \ref{thm}]\label{aa}
Let $X$ be a smooth complex projective variety of dimension $n$ and $D$ an effective $\mathbb{Q}$-divisor. Let $f:Y\rightarrow X$ be a log resolution of $(X,D)$  which is isomorphic over $X\setminus D$. Denote by $ E$ the support of $f^*D$. 

(1) If $\mathcal{A}$ is a line bundle on $X$ such that 
$\mathcal{A}-D$ is ample, then
$$H^q\big(Y,\Omega^p_Y(\log  E)\otimes f^*\mathcal{A}(-\lceil f^*D\rceil)\big)=0$$
for any $p+q>n$.

(2) If there exists an ample effective divisor on $X$ with support contained in $\text{Supp}(D)$, then (1) also holds for line bundle $\mathcal{A}$ such that $\mathcal{A}-D$ is nef.
\end{thm}
\begin{proof}
(1) Choose a sufficiently large integer $l$ such that $l\cdot D$ is an integral divisor and $\mathcal{A}^l(-l\cdot D)$ is very ample. Let $H$ be a divisor on $X$ corresponding to a sufficiently general section of $\mathcal{A}^l(-l\cdot D)$ such that $f^*H+ E$ is a reduced simple normal crossing divisor on $Y$. 

Since $H$ is very ample, $X\setminus H$ is affine, which implies that $X\setminus (H\cup D)$ is affine. As $f$ is isomorphic outside $D$, we see that $Y\setminus (f^*H\cup  E)$ is also affine.

Notice that  
$$f^*\mathcal{A}^l \simeq\mathcal{O}_Y(l\cdot f^*D+f^*H).$$ 
For $i=0,\cdots,l-1$, we denote
$$f^*\mathcal{A}^{(i)}=f^*\mathcal{A}^i(-\lfloor \frac{i}{l}\cdot (l\cdot f^*D +f^*H)\rfloor).$$
By \cite[Theorem 3.2]{EV92}, $(f^*\mathcal{A}^{(i)})^{-1}$ 
has an integrable logarithmic connection along $ E+f^*H$ such that the Hodge-to-de Rham spectral sequence degenerates at $E_1$ and the eigenvalues of the residue map belong to $[0,1)$ (hence condition $(*)$ of (2.8) in \cite{EV92} holds).
Since $Y\setminus ( E\cup f^*H)$ is affine, by Artin vanishing (see e.g. \cite[Corollary 5.2.18]{Dim04}),  
$$H^q\big(Y,(f^*\mathcal{A}^{(i)})^{-1}\otimes \Omega^p_Y(\log( E+f^*H))\big)=0$$
for $p+q>n$ and $i=0,\cdots,l-1$.
In particular, 
\begin{align*}
(f^*\mathcal{A}^{(l-1)})^{-1}&=f^*\mathcal{A}^{1-l}(\lfloor \frac{l-1}{l}\cdot (l\cdot f^*D +f^*H)\rfloor)\\
&=f^*\mathcal{A}^{1-l}(l\cdot f^*D +f^*H+\lfloor-\frac{1}{l}\cdot (l\cdot f^*D +f^*H)\rfloor)\\
&\simeq f^*\mathcal{A}(-\lceil \frac{1}{l}\cdot(l\cdot f^*D +f^*H)\rceil)\\
&=f^*\mathcal{A}(-\lceil f^*D \rceil- f^*H).
\end{align*}
The last equality holds because $f^*H$ is reduced and has no common components with $ E=(f^*D)_{red}$. Therefore, we have 
\begin{align}\label{eqA}
H^q\big(Y,f^*\mathcal{A}(-\lceil f^*D \rceil-f^*H)\otimes \Omega^p_Y(\log( E+f^*H))\big)=0
\end{align}
for any $p+q>n$.
Consider the short exact sequence
$$0\rightarrow \Omega^p_Y(\log( E+f^*H))(-f^*H)\rightarrow \Omega^p_Y(\log E)\rightarrow \Omega^p_{f^*H}(\log E\mid_{f^*H})\rightarrow 0.$$
(see \cite[Properties 2.3(c)]{EV92}). Applying $\otimes f^*\mathcal{A}(-\lceil f^*D \rceil)$, we obtain a short exact sequence
\begin{align*}
0\rightarrow \Omega^p_Y(\log( E+f^*H))\otimes f^*\mathcal{A}(-\lceil f^*D \rceil-f^*H)\rightarrow \Omega^p_Y(\log( E)\otimes f^*\mathcal{A}(-\lceil f^*D \rceil)\\
\rightarrow \Omega^p_{f^*H}(\log E\mid_{f^*H})\otimes f^*\mathcal{A}(-\lceil f^*D \rceil)\mid_{f^*H}\rightarrow 0.
\end{align*}
Since $(f^*H, E\mid_{f^*H})$ is a log resolution of $(H,D\mid_{H})$ and $(\mathcal{A}-D)\mid_{f^*H}$ is ample,
by inductive assumption on the dimension of $X$, we have 
\begin{align}\label{eqB}
H^q(f^*H,\Omega^p_{f^*H}(\log E\mid_{f^*H})\otimes f^*\mathcal{A}(-\lceil f^*D \rceil)\mid_{f^*H})=0
\end{align}
for $p+q>n-1$.
(\ref{eqA}) and (\ref{eqB}) imply that 
$$H^q(Y,\Omega^p_Y(\log  E)\otimes f^*\mathcal{A}(-\lceil f^*D \rceil)=0$$
for $p+q>n$.

(2) Let $H$ be an ample effective divisor on $X$ with support contained in $\text{Supp}(D)$. For $0<\varepsilon\ll 1$, we have
$$\lceil f^*(D-\varepsilon H)\rceil=\lceil f^*D\rceil.$$
and $$\text{Supp}\big(f^*(D-\varepsilon H)\big)=\text{Supp}(f^*D)= E.$$
Since $\mathcal{A}-D$ is nef, $\mathcal{A}-(D-\varepsilon H)$ is ample. Therefore the assertion in (2) follows from that in (1).
\end{proof}

\section{vanishing theorem for Hodge ideals of $\mathbb{Q}$-divisors}

With the help of Theorem \ref{main}, using similar arguments in the proof of \cite[Theorem F]{MP16}  and \cite[Theorem 12.1]{MP19a}, now we can eliminate the global assumption in the vanishing theorem for Hodge ideals of $\mathbb{Q}$-divisors. 

\begin{thm}[=Theorem \ref{thm1}]\label{van}
Let $X$ be a smooth projective complex variety of dimension $n$ and $D$ an effective $\mathbb{Q}$-divisor on $X$. Denote by $Z$ the support of $D$. For some $k\geq 0$, assume that $(X,D)$ is reduced $(k-1)$-log-canonical, i.e. $I_0(D)=\cdots=I_{k-1}(D)=\mathscr{O}_X(Z-\lceil D\rceil)$.

(1) Let $\mathcal{L}$ be a line bundle such that $\mathcal{L}-D$ is ample. If $\mathcal{L}((p+1)Z-\lceil D\rceil)$ is ample for all $0\leq p \leq k-1$, then
$$H^i(X,\omega_X \otimes \mathcal{L}(kZ)\otimes I_k(D))=0$$
for all $i \geq 2$. Moreover,
$$H^1(X,\omega_X \otimes \mathcal{L}(kZ)\otimes I_k(D))=0$$
holds if 
\begin{align}\label{N}
H^{k-p}\big(X,\Omega_X^{n-k+p}\otimes \mathcal{L}((p+1)Z-\lceil D\rceil)\big)=0
\end{align}
 for all $0\leq p \leq k-1$.

(2) If there exists an ample effective divisor with support contained in $Z=\text{Supp}(D)$, then (1) also holds for $\mathcal{L}$ such that $\mathcal{L}-D$ is nef.
\end{thm}
\begin{proof}
For $k=0$, we have $$\omega_X \otimes \mathcal{L}(kZ)\otimes I_k(D)\simeq \mathcal{M}_k(D)\otimes \mathcal{L}.$$
For $k\geq 1$, note that $I_{k-1}(D)=\mathscr{O}_X(Z-\lceil D \rceil)$, we have a short exact sequence
$$0\rightarrow \omega_X\big((k-1)Z\big)\otimes \mathcal{L}(Z-\lceil D\rceil)\rightarrow \omega_X(kZ) \otimes \mathcal{L} \otimes I_k(D)\rightarrow \mathcal{M}_k(D)\otimes \mathcal{L}\rightarrow 0.$$
Since $\mathcal{L}(kZ-\lceil D\rceil)$ is ample, by using Kodaira vanishing, we see that the vanishing we are aiming for is equivalent to the same statement for 
$$H^i(X,\mathcal{M}_k(D)\otimes \mathcal{L}).$$

\begin{cla}\label{cla}
Let $\mathcal{L}$ be a line bundle on $X$ such that $\mathcal{L}-D$ is ample (or nef when the condition in Theorem \ref{van}(2) is satisfied).
Let $r$ be a positive integer. If $$H^i(X,\mathcal{M}_p(D)\otimes \wedge^{k-p}T_X \otimes \mathcal{L})=0$$ for any $p<k$ and any $i\geq k-p+r-1$, then
$$H^i(X,\mathcal{M}_k(D)\otimes \mathcal{L})=0 \quad \text{for all } i\geq r$$
\end{cla}
\begin{proof}[Proof of the claim]
By definition, $\text{Sp}_{k-n}(\mathcal{M}_{\bullet}(D))=$
$$[\mathcal{M}_{k-n}(D)\otimes\wedge^n T_X\stackrel{\widetilde{d}}{\longrightarrow} \mathcal{M}_{k-n+1}(D)\otimes\wedge^{n-1} T_X\stackrel{\widetilde{d}}{\longrightarrow} \cdots \stackrel{\widetilde{d}}{\longrightarrow}\mathcal{M}_{k}(D)\otimes\wedge^0 T_X]$$
placed in degrees $-n,-n+1,\cdots,0$. Denote by $E^{\bullet}$ the complex $\text{Sp}_{k-n}(\mathcal{M}_{\bullet}(D))\otimes \mathcal{L}$. By Theorem \ref{main}, 
$$\mathbf{H}^i(X,E^{\bullet})=0 \quad \text{for all } i>0.$$
For any integer $i$, let $E_i^{\bullet}$ be the subcomplex of $E^{\bullet}$ given by $E_i^p=0$ for $p<i$ and  $E_i^p=E^p$ for $p\geq i$. Then 
$$E_0^{\bullet}=\mathcal{M}_k(D)\otimes \mathcal{L}\quad 
\text{ and } \quad E_{-n}^{\bullet}=E^{\bullet}.$$
We have a short exact sequence
$$0\rightarrow E^{\bullet}_{-p+1}\rightarrow E^{\bullet}_{-p} \rightarrow \mathcal{M}_{k-p}(D)\otimes \wedge^{p}T_X\otimes \mathcal{L}[p]\rightarrow 0.$$
Because of the hypothesis in the claim, $$H^i(X,\mathcal{M}_{k-p}(D)\otimes \wedge^pT_X \otimes \mathcal{L}[p])=0$$ for $p>0$ and $i\geq r-1$. It follows  that
$$\mathbf{H}^i(X,E^{\bullet}_{-p+1})\simeq \mathbf{H}^i(X,E^{\bullet}_{-p}) \quad \text{for } i\geq r \text{ and } p>0.$$
Therefore,
$$H^i(X,\mathcal{M}_k(D)\otimes \mathcal{L})=\mathbf{H}^i(X,E^{\bullet}_0)=\mathbf{H}^i(X,E^{\bullet}_{-n})=\mathbf{H}^i(X,E^{\bullet})=0$$
for $i\geq r$.
\end{proof}
Now we return to the proof of Theorem \ref{van}. Consider the short exact sequence
$$0\rightarrow \omega_X\big((p-1)Z\big)\otimes I_{p-1}(D)\rightarrow  \omega_X(pZ)\otimes I_p(D)\rightarrow \mathcal{M}_p(D)\rightarrow 0.$$
Applying $\otimes \wedge^{k-p}T_X\otimes \mathcal{L}$, we obtain 
\begin{equation}\label{exact}
\begin{aligned}
0\rightarrow \Omega^{n-k+p}_X\otimes \mathcal{L}\big((p-1)Z\big)\otimes I_{p-1}(D)\rightarrow  \Omega^{n-k+p}_X\otimes \mathcal{L}(pZ)\otimes I_{p}(D)\\
\rightarrow \mathcal{M}_p(D)\otimes \wedge^{k-p}T_X\otimes \mathcal{L}\rightarrow 0.
\end{aligned}
\end{equation}
Since 
$$I_p(D)= \begin{cases}
\mathscr{O}_X(Z-\lceil D \rceil), & 0\leq p<k\\
0, & p<0
\end{cases}
$$
and $\mathcal{L}\big((p+1)Z-\lceil D\rceil\big)$ is ample for $0\leq p<k$, we see that
$\mathcal{L}(pZ)\otimes I_p(D)$ is a ample bundle or a zero bundle for $p<k$. Therefore,
\begin{equation}\label{22}
\begin{aligned}
&H^i\big(\Omega^{n-k+p}_X\otimes \mathcal{L}((p-1)Z)\otimes I_{p-1}(D)\big)\\
=&H^i(\Omega^{n-k+p}_X\otimes \mathcal{L}(pZ)\otimes I_{p}(D))=0
\end{aligned}
\end{equation}
for $i\geq k-p+1$ and $p<k$. By taking the long exact sequence  in cohomology groups associated with (\ref{exact}), we obtain the same statement for
$$H^i(X,\mathcal{M}_p(D)\otimes \wedge^{k-p}T_X\otimes \mathcal{L}).$$
By Claim \ref{cla} (taking $r=2$), we have
$$H^i(X,\mathcal{M}_k(D)\otimes \mathcal{L})=0 \quad \text{for } i\geq 2.$$ 
Moreover, if $H^{k-p}\big(X,\Omega_X^{n-k+p}\otimes \mathcal{L}((p+1)Z-\lceil D\rceil)\big)=0$ for all $0\leq p <k$, then we have 
\begin{align}\label{33}
H^{k-p}(X,\Omega^{n-k+p}_X\otimes \mathcal{L}(pZ)\otimes I_{p}(D))=0\quad \text{for all } p<k.
\end{align}
(\ref{exact}), (\ref{22}) and (\ref{33}) imply that 
$$H^i(X,\mathcal{M}_p(D)\otimes \wedge^{k-p}T_X\otimes \mathcal{L})=0$$
for $i\geq k-p$ and $p<k$. By Claim \ref{cla} (taking $r$=1), we have 
$$H^i(X,\mathcal{M}_k(D)\otimes \mathcal{L})=0 \quad \text{for } i\geq 1.$$ 
\end{proof}

\begin{rem}[\textbf{Toric varieties}]
As in \cite[Corollary 25.1]{MP16} and \cite[Remark 12.3]{MP19a}, if $X$ is a toric variety, with the help of the Bott-Danilov-Steenbrink vanishing theorem, we can remove the Nakato-type vanishing assumption (\ref{N}) in the above theorem. In this setting, Dutta \cite[Theorem A]{Dut18} proves a stronger vanishing theorem under the global assumption. Similarly, the global assumption in her statement can be removed.

%for any ample line bundle $\mathcal{A}$ on a smooth projective toric variety $X$ one has
%$$H^i(X,\Omega^j_X\otimes \mathcal{A})=0 \quad \text{for all } j\geq 0 \text{ and } i>0.$$
\end{rem}

\begin{rem}[\textbf{Projective space, abelian varieties}]
As in \cite[Theorem 25.3 and 28.2]{MP16} and \cite[Variant 12.5 and 12.6]{MP19a}, Theorem \ref{van} on $\mathbb{P}^n$ and abelian varieties holds for much weaker hypotheses. 
\end{rem}

On $\mathbb{P}^n$, the global assumption is equivalent to the degree $d$ of the $\mathbb{Q}$-divisor $D$ is an integer. So we can remove the condition that $d$ is integer in \cite[Variant 12.5]{MP19a}.

\begin{var} Let $D$ be an effective $\mathbb{Q}$-divisor of degree $d>0$ on $\mathbb{P}^n$ ($d$ is a rational number). Denote by $Z$ the support of $D$. For any integer $\ell\geq d-n-1$, we have
$$H^i(\mathbb{P}^n, \mathscr{O}_{\mathbb{P}^n}(\ell)\otimes \mathscr{O}_{\mathbb{P}^n}(kZ)\otimes I_k(D))=0  \quad \text{for any } i>0 \text{ and } k\in\mathbb{Z}.$$
\end{var}
\begin{proof}
Since $$\mathcal{M}_p(D)\simeq\mathscr{O}_{\mathbb{P}^n}(-n-1)\otimes \frac{\mathscr{O}_{\mathbb{P}^n}(pZ)\otimes I_p(D)}{\mathscr{O}_{\mathbb{P}^n}((p-1)Z)\otimes I_{p-1}(D)},$$
to prove the assertion in the corollary, it suffices to prove that for any integers $p\leq k$ and $q \geq d$
$$H^i(\mathbb{P}^n,\mathcal{M}_p(D)\otimes \mathscr{O}_{\mathbb{P}^n}(q))=0 \quad \text{for all } i>0.$$
We prove it by increasing induction on $k$. If $k\leq 0$, $\mathcal{M}_p(D)=0$ for all $p\leq k$, so there is nothing further to prove. Otherwise, assume that the assertion holds for $k-1$.
Note that when $0\leq j<n$, there is a long exact sequence (the Koszul resolution of $\wedge^jT_{\mathbb{P}^n}$)
\begin{align}\label{la}
0\rightarrow \bigoplus \mathscr{O}_{\mathbb{P}^n} \rightarrow \bigoplus\mathscr{O}_{\mathbb{P}^n}(1)\rightarrow \cdots \rightarrow\bigoplus\mathscr{O}_{\mathbb{P}^n}(j)\rightarrow \wedge^{j} T_{\mathbb{P}^n}\rightarrow 0.
\end{align}
and when $j=n$, we have $\wedge^{n} T_{\mathbb{P}^n}\simeq \mathscr{O}_{\mathbb{P}^n}$. Therefore, the inductive assumption implies that for any integers $p\leq k-1$ and $q \geq d$,
$$H^i(X,\mathcal{M}_p(D)\otimes \wedge^{k-p}T_X \otimes \mathscr{O}_{\mathbb{P}^n}(q))=0 \quad \text{for all } i>0.$$
Since $D$ is ample and $\mathscr{O}_{\mathbb{P}^n}(q)-D$ is nef for $q\geq d$, the desired assertion follows from Claim \ref{cla} (taking $r=1$).
\end{proof}

On abelian varieties, we have 

\begin{var} Let $X$ be an abelian variety and $D$ be an effective $\mathbb{Q}$-divisor on $X$. Denote by $Z$ the support of $D$. If $\mathcal{L}$ is a line bundle on $X$ such that $\mathcal{L}-D$ is ample. Then
$$H^i(X,\mathcal{L}(kZ)\otimes I_k(D))=0 \quad \text{for any } i>0 \text{ and } k\in\mathbb{Z}.$$
Moreover, if $D$ is ample, then the conclusion holds for $\mathcal{L}$ such that $\mathcal{L}-D$ is nef.
\end{var}
\begin{proof}
The assertion follows from the same argument in the proof of the previous variant, once we replace (\ref{la}) by the fact that $\wedge^j T_X$ is a trivial vector bundle for any $j$.
\end{proof}

% ----------------------------------------------------------------

\begin{thebibliography}{AGLV93}
\bibitem[AN54]{AN54} Y. Akizuki, S. Nakano, \emph{Note on Kodaira-Spencer's proof of Lefschetz's theorem},
Proc. Jap. Acad., Ser A \textbf{30} (1954), 266-272.
\bibitem[Dim04]{Dim04} A. Dimca, \emph{Sheaves in topology}, Universitext, Springer-Verlag, Berlin, 2004.
\bibitem[Dut20]{Dut18} Y. Dutta, \emph{Vanishing for Hodge ideals on toric varieties}, Math. Nachr. \textbf{293} (2020), no. 1, 79-87.
\bibitem[EV92]{EV92} H. Esnault and E. Viehweg, \emph{Lectures on vanishing theorems}, DMV seminarband \textbf{20}, Birkhauser. (1992).
\bibitem[Laz04]{Laz04} R. Lazarsfeld, \emph{Positivity in algebraic geometry I}, Ergebnisse der Mathematik und ihrer Grenzgebiete, vol. \textbf{48}, Springer-Verlag, Berlin, 2004.
\bibitem[MP19a]{MP16} M. Musta\c t\u a and M. Popa, \emph{Hodge ideals}, Memoirs of the AMS \textbf{262} (2019).
\bibitem[MP19b]{MP19a} M. Musta\c t\u a and M. Popa, \emph{Hodge ideals for $\mathbb{Q}$-divisors: birational approach}, J. \' Ec. polytech. Math. \textbf{6} (2019), 283-328.
\bibitem[MP20a]{MP18} M. Musta\c t\u a and M. Popa, \emph{Hodge ideals for $\mathbb{Q}$-divisors, V-filtration, and minimal exponent}, Forum of Mathematics, Sigma \textbf{8} (2020), e19, 41pp.
\bibitem[MP20b]{MP19b} M. Musta\c t\u a and M. Popa, \emph{Hodge filtration, minimal exponent, and local vanishing}, Invent. Math. \textbf{220} (2020), 453-478.
\bibitem[Sai88]{Sai88} M. Saito, \emph{Modules de Hodge polarisables}, Publ. Res. Inst. Math. Sci. \textbf{24} (1988), no. 6, 849-995.
\bibitem[Sai90]{Sai90} M. Saito, \emph{Mixed Hodge modules}, Publ. Res. Inst. Math. Sci. \textbf{26} (1990), no. 2, 221-333.

\bibitem[Sai07]{Sai07} M. Saito, \emph{Direct image of logarithmic complexes and infinitesimal invariants of cycles},
Algebraic cycles and motives. Vol. 2, London Math. Soc. Lecture Note Ser. Vol. 344,
Cambridge Univ. Press, Cambridge, 2007, pp. 304-318.
\end{thebibliography}
\end{document}